\documentclass[11pt]{amsart}
\usepackage{amsmath,amssymb,enumerate}
\usepackage{latexsym}
\usepackage{amsfonts}
\usepackage {amsbsy}
\usepackage {amsmath}
\usepackage{amssymb}
\newtheorem{defn}{Definition}[section]
\newtheorem{thm}[defn]{Theorem}

\newtheorem{lem}[defn]{Lemma}
\newtheorem{cor}[defn]{Corollary}

\newcommand \C{\mathbb C}
\newcommand \N{\mathbb N}

\newcommand \R{\mathbb R}
\newcommand \B{\mathbb B}

\newcommand \U{\mathbb U}
\newcommand \T{\mathbb T}

\newcommand \sub{\subset}

\newcommand \ove{\overline}

\numberwithin{equation}{section}

 \title{ H\"older regularity of generic manifold}
 \author{Azimbay Sadullaev and Ahmed Zeriahi}
\begin{document}
\maketitle
\begin{center}
{\it Dedicated to Professor J\'ozef Siciak for his $85^{th}$ birthday}
\end{center}

\vskip  0.3 cm

\noindent{\bf Abstract :} In this paper we study H\"older continuity of the pluricomplex Green function with logarithmic growth at infinity  of a smooth generic submanifold of $\C^n$. In particular we prove that the pluricomplex Green function of any $C^2$-smooth generic compact submanifold  of $\C^n$ (without boundary) is Lipschitz continuous in $\C^n$. 



 \section{Introduction and statement of the main result}
 Real $m-$planes $\Pi \sub \C^n,
\,\,  \text{dim}_{R}\Pi=m, \,\,\, m \in \N^+,$ which are not contained in
any proper complex subspace of $\C^n$ are important in complex
analysis and pluripotential theory. The $\C-$hull of such plane
$\Pi$ is equal to all $\C^n$ i.e. $\Pi + J \Pi = \C^n$ ($J$ is the
standard complex structure on $\C^n$) and any non empty open
subset of $\Pi$ is non pluripolar in $\C^n$. Such planes are
called {\it generic} (real) subspaces of $\C^n$.
 Correspondingly, a real smooth submanifold $M \subset \C^n$ is said to be {\it generic} if for each $z \in M$,
 its real tangent space $T_z M $ is a generic subspace of $\C^n$  i.e. $T_z M + J T_z M = \C^n$.
 Such submanifold has real dimension $m \geq n$.
The case of minimal dimension $\text{dim} M = n$ is the most relevant for our concern. In
this case for each $z \in M$, the tangent space $T_z M$ does not
contain any complex line i.e. $ T_z M \cap J T_z M = \{0\}$ and
$M$ is said to be {\it totally real}.

Observe that any smooth Jordan curve in $\C$ is totally real, hence any product of $n$ smooth Jordan curves in $\C$ is a smooth compact totally real submanifold of dimension $n$ in $\C^n$. Moreover the class of  smooth compact totally real submanifolds of dimension $n$ in $\C^n$ is  stable under small $C^{2}$-perturbations.
 
 Generic submanifolds of $\C^n$ play an important role in Complex Analysis and Pluripotential Theory (see [Pi74], [KC76], [Sa76], [C92],  [EW10], [SZ12]).
 
  --------------------------------------------------------------------\\
*) The first author was partially supported by the fundamental research of Khorezm
Mamun Academy,Grant $\Phi 4-\Phi A-0-16928$.\\

In our previous paper [SZ12], we used the method of attached analytic discs to investigate
  non plurithiness of generic submanifolds of $\C^n$. We proved in [SZ12], that subsets of full measure in a generic $C^2-$smooth submanifold are non-plurithin at any point.

  Here we continue our investigations concerning the {\it pluripotential} properties  ({\it pluripolarity, pluriregularity }) of generic submanifolds in $\C^n$ by studing H\"older continuity of their pluricomplex Green functions.

All these properties can be expressed in terms of the pluricomplex Green function defined as follows.

Given a (bounded) subset $E \Subset \C^n$, we define its pluricomplex Green function as follows:
 $$
 V_E (z) := \sup \{u (z) : u \in \mathcal L (\C^n), u|E \leq 0\},
 $$
 where $\mathcal L (\C^n)$ is the  Lelong class of $psh$ functions $u$ in $\C^n$ with logarithmic growth at infinity i.e. $ \sup \{u (z) - \log^+ (z) : z \in \C^n\} < + \infty$ (see [Sic62], [Sic81], [Za76], [Kl]).

 Our main result is the following.
 \vskip 0.2 cm
\noindent {\bf Main theorem.} {\it Let $M \subset \C^n$ be a $C^2$-smooth generic compact
submanifold without boundary. Then its pluricomplex Green function $V_ M$ is Lipschitz continuous in $\C^n$.}
\vskip 0.3 cm

 This theorem is concerned with  compact submanifolds without boundary.
 In Section 3,  we will consider the more general case of a $C^2$-smooth generic submanifold and  prove that its extremal function is Lipschitz near each of its compact subsets (see Theorem 5.1).
 In the last section we consider the case of  a compact $C^2$-smooth generic submanifold with boundary and discuss the H\"older continuity property of its pluricomplex Green function.

 From Lipshitz continuity or more generally the H\"older continuity of the pluricomplex Green function $V_E^*(z)$ of a compact set $E \subset \C^n$, it follows that the compact set $E$ satisfies the following Markov's inequality: there exists positive constants $A, r > 0$ such that
  $$
  \Vert \nabla P (z) \Vert_E  \leq A d^r \Vert P (z)
  \Vert_E,\,\,z\in \C^n
  $$
  for any polynomial $P$ of degree $d$.

  This inequality plays an important role in approximation theory, gives sharp inequalities for polynomials and is useful for
  constructing continuous extension operators for smooth functions from subsets of $\R^n$ to $\C^n$ (see [PP86], [Ze93]).
On the other hand, Complex Dynamic gives a lot of examples of compact subsets for which the pluricomplex Green function $V_E^*(z)$  is H\"older continuous (see [FS92, K95, Ks97]).

  More recently an important result of C.T. Dinh, V. A. Nguyen  and N. Sibony shows that the Monge-Amp\`ere measure of
  a H\"older continuous plurisubharmonic function is a {\it moderate measure} (see [DNS]). In particular the equilibrium  Monge-Amp\`ere
  measure $\mu_E := (dd^c V_E)^n$ of a compact subset whose pluricomplex Green function $V_E^*(z)$  is H\"older continuous is a moderate measure,  which means  that it satisfies the following uniform version of Skoda's integrability theorem: for any compact family $\mathcal U$
    of $psh$ functions in a neighborhood of  a given ball $\B \subset \C^n$, there exists  $\varepsilon > 0$ and a
     constant $C > 0$ such that

$$
  \int_{\B} e^{- \varepsilon u} d \mu_E \leq C, \forall u \in \mathcal U.
  $$
  From this property, it follows that the equilibrium measure $\mu_E$ is "well dominated" by the Monge-Amp\`ere capacity (see [Ze01]), in the sense that for any given ball $\B \subset \C^n$, there is a constant $A > 0$ such that for any Borel set $S \subset \B$,

  $$
  \mu_E (S) \leq A \exp \left(- A \text{cap}_{\B} (S)^{- 1 \slash n}\right),
  $$
 where $ \text{cap}_{\B} (S)$ is the Monge-Amp\`ere capacity ([BT82]).

 This property turns out to play an important role in the theory of complex Monge-Amp\`ere equations as was discovered by S. Kolodziej (see [Ko98], [Ce98],[GZ07]).

\vskip 0.3 cm
\noindent {\bf {Acknowledgments:}}  1. It is a pleasure for us to dedicate this paper to Professor J\'ozef Siciak for his $85^{th}$ birthday. His remarkable achievement in Pluripotential Theory and Polynomial Approximation has been a great inspiration for us.

2.  We would like to thank Evgeniy Chirka and  Norm Levenberg for interesting and helpful discussions on these problems. \\

3. This work was started during the visit of the two authors to ICTP in June 2012. The
authors would like to thank warmly this Institute for providing excellent conditions
for mathematical research and specially Professor Claudio Arezzo for the invitation.

4. We would like to the thank the referee for his remarks.

\section{Definitions and preliminaries}

 Let us recall the following definitions

  \begin{defn} 1. We say that a subset $P\subset \C^n$ is {\it pluripolar} if
  there is a plurisubharmonic ($psh$) function $u: u \not\equiv
  -\infty$ but $u|_P\equiv -\infty.$ \\
  2. We say that $E$ is
  {\it pluriregular} if its pluricomplex Green function satisfies $ V^*_E(z)|_E\equiv 0$ i.e. $V_E = V_E^*$ on $\C^n$.
 \end{defn}
 Observe that any pluriregular set is non-pluripolar.
  It is well-known that if $E$ in non-pluripolar then $V_E^* \in \mathcal L (\C^n)$. Moreover  if  $E$ is   pluriregular compact set then $V_E = V_E^*$ is continuous in $\C^n$ (see [Sic81]).

  On the other hand, we know from ([BT82]) that if $E$ is non-pluripolar then the locally bounded $psh$ function $V_E^*$
   satisfies the following complex Monge-Amp\`ere equation
 $$
  (dd^c V_E)^n = 0, \ \ \text{on} \ \ \C^n \setminus \overline{E},
 $$
 which means that the equilibrium measure of $E$ defined as
 $$
  \mu_E :=(dd^c V_E^*)^n
 $$
 is a Borel measure supported in the closed set $\overline E$.

  Here we will introduce the following important notion.
  \begin{defn}
   We say that a set $E$ is $\Lambda_{\alpha}$-pluriregular, $\alpha >0$, if for every compact
  $K \subset E$ there exist a constant $A = A_K >0$ and a neighborhood $O = O_K$  of $K$ such that
  \begin {equation} \label{eq:H1}
  V_E^*(z) \leq A d^{\alpha} (z,K), \ \forall z \in O,
  \end {equation}
  where $d$ is the Euclidean distance in $\C^n$.
  \end{defn}
  Roughtly speaking, this definition means that the pluricomplex Green function $V_E$ of the set $E$ is H\"older continuous near any compact subset $K \subset E$.
  The following observation, which is essentially due to Z. B{\l}ocki, shows that if the set itself $E$ is compact, the definition means that its pluricomplex Green function is H\"older continuous (see [Sic97]).
  \begin{lem}
   If $E \subset \C^n$ is a $\Lambda_{\alpha}$-pluriregular compact set then its pluricomplex Green function $V_E$ is H\"older continous
  of order $\alpha$ globally in $\C^n$  i.e. for any $z, w \in  \C^ n$, we have
  $$
 \vert  V_E (z) - V_E (w) \vert \leq A \vert z - w\vert^ {\alpha}.
  $$
  \end{lem}
  \begin{proof}
  Observe that $V_E^*$ has a logarithmic growth at infinity. Therefore if $E$ is a $\Lambda_{\alpha}$-pluriregular compact set then
 its pluricomplex Green function $V_E$ satisfies (\ref{eq:H1}) for all $z\in \C^n$ i.e. for some constant $A > 0$ we have
  \begin {equation} \label{eq:H2}
  V_E(z) \leq A d^{\alpha} (z,E), \ \forall z \in \C^n.
  \end {equation}
  To prove that $V_E$ is H\"older continous
  of order $\alpha$ globally in $\C^n$, fix $h\in \C^n$ such that $\vert h\vert < \delta$
  and observe, that for any $z \in E$, $d (z+h,E) \leq \delta^{\alpha}$, which implies by the H\"older condition (\ref{eq:H2})  that  for any $z \in E$, $V_E (z + h) \leq  A \delta^{\alpha}$. Therefore the function defined by
    $u (z) := V_E (z + h) - A \delta^{\alpha} $  is a plurisubharmonic function such that  $u \in \mathcal L (\C^ n)$
    and $u \leq 0$ on $E$. By the definition of $V_E$, we conclude that $u \leq V_E (z)$ for any $z \in \C^ n$, which
    implies that $V_E$ is  H\"older continuous.

  \end{proof}

 \section{Analytic discs attached to generic manifolds}

\subsection{Construction of attached analytic discs}
 Let $\U := \{\zeta \in \C : \vert \zeta \vert < 1\}$ be the open unit disc
and $\T := \partial \U$ the unit circle.
 An analytic disc of $\C^n$ is a continuous function $f : \overline{\U} \longrightarrow \C^n$, which is holomorphic
  on $\U$. Let $M \subset \C^n$ be a given subset of $\C^n$ and  $\gamma \subset \overline{U} $ a given connected subset
   of the closed disc $\bar \U$. We say that the analytic disc $f$ is attached to  $M$ along  $\gamma$  if
    $f (\gamma) \subset M$.

 If $f : \ove{\U} \longrightarrow \C^n$ is an analytic disc and $F$ is a holomorphic function on a neighboorhood $D$
  of $f (\ove{\U})$, then $F \circ f$ is a holomorphic function on the unit disc $\U$. If $u$ is a plurisubharmonic
  function on $D$, then $u \circ f$ is a subharmonic function on $\U$.
 Therefore analytic discs enable us to reduce multidimentional complex problems to corresponding one dimensional
 complex problems.

  In the proof of our theorem, we need a smooth family of analytic disks.  We will use  Bishop's   equation
 for construct such a family (see [B65], [Pi74]). Let M be a
totally real submanifold of dimension  $n$ given locally by the
following equation
$$
M: = \left\{ {z = x + y \in B \times \R^n :y = h\left( x \right)}
\right\},
$$
where  $ B \subset \R^n $
 is a ball of center 0 and $
h : B \rightarrow \R^n$ a smooth map, such that
$$ h\left( 0 \right) = 0  \ \ \text{and} \ \
Dh\left( 0 \right) = 0.
$$

 Let   $v\left( \tau  \right) : \T \rightarrow \R^+ $ a $C^\infty $
 function on the unit circle $\T$ such that
 $$
v|_{\left\{e^{i \theta} \, : \, \, \,  \theta \in (0,\pi)\right\}}  = 0 \ \
\text{and} \  \ v|_{\left\{e^{i \theta} \,  :  \, \, \, \theta \in (\pi ,2\pi)
\right\}}
> 0 .
$$

  Assume that there exists a continuous mapping $X : \T \to \R^n $ which is a solution of the following Bishop equation

  \begin{equation}
X\left( \tau  \right) = c - \Im \left( {h \circ X + tv}
\right)\left( \tau  \right), \, \, \,\tau  \in \T,
\end{equation}
where $ \left( {c,t} \right) \in Q = Q_c \times Q_t \subset \R^n  \times \R^n $ is
a fixed parameter and $\Im $ is the harmonic conjugate operator
defined by the Schwarz integral formula

  \begin{equation}
\Im \left( X \right)\left( \zeta  \right) = \frac{1} {{2\pi
}}\int\limits_T {X\left( \tau  \right)\operatorname{Im}
\frac{{e^{i\tau }  + \zeta }} {{e^{i\tau }  - \zeta }}d\tau }
\,\,,\,\,\,\,\zeta  = re^{i\theta } ,
\end{equation}
normalized by the condition
$$ \Im X\left( 0 \right) = 0.
$$
  We will consider the unique harmonic extension $
X\left( \zeta  \right) $
 of the mapping $X $ to the unit disk $\U.$ Then  the following mapping
\begin{equation}
\begin{array}{ll} \Phi \left( {c,t,\zeta } \right): = X\left(
{c,t,\zeta } \right) + i\left[ {h^{*}\left( {c,t,\zeta } \right) +
tv\left(
\zeta  \right)} \right] = \\
= c + i\left\{ {h^{*}\left({c,t,\zeta } \right) + tv\left( \zeta
\right) + i\Im \left[ {h^{*}\left( {c,t,\zeta } \right) + tv\left(
\zeta \right)} \right]} \right\}
\end{array}
\end{equation}
provides a  family of analytic disks  $ \Phi \left( {c,t,\zeta }
\right):\overline \U  \to {\Bbb C}^n $
 such that
      \begin{equation}
\forall \left( {c,t} \right) \in Q\,,\,\forall \, \tau  \in \gamma
\,,\,\Phi \left( {c,t,\tau } \right) \in M.
\end{equation}
Here  $ X\left( {c,t,\zeta } \right) ,\,\, h^{*}\left( {c,t,\zeta
} \right) \,\,\text{and}\,\,v\left( \zeta \right) $
 are harmonic extensions of  $
X\left( {c,t,\tau } \right), \\ h \circ X\left( {c,t,\tau }
\right)\,\,\,\text{and}\,\,\,v\left( \tau  \right) $
 to the unit disk $\U$ respectively.

  We need a smooth family of disks $\Phi \left( {c,t,\zeta} \right)$.
Many constructions of analytic discs attached to generic manifolds along a part the circle have been given by many
 different authors, depending on the smoothness properties of the manifold (see [Pi74],  [75], [Sa76], [C92]).
 The most general and sharp result was proved by B.Coupet [C92]:

\vskip 0.3 cm
 \noindent {\bf Theorem} ([C92]). {\it Let $p>2n+1, q\geq 1$ be integers
 and $h\in C^q (B)$. Then there exist a constant $\delta_0 >0$ independing on $h$ and $p$ such, that
 for arbitrary $C^q$-smooth mapping $k(c,t,\tau): \R^{2n+1}\to \R^n, \,\,$ with compact support and
 $\parallel{k}\parallel_{W^{q,p}}\,  \leq \,\delta_0$, the equation
 \begin {equation}
   u=-\Im (h \circ u)+ k
  \end {equation}
 has a unique solution $u\in W^{q,p}(\T \times \R^{2n}).$

 Moreover, the harmonic extensions of $u$ and $h\circ u$ to  the unit disk $\U$
belong  to $C^q{(\U \times \R^{2n})}.$}
\vskip 0.3 cm

Let now $h\in C^q (\B)$. Observe that Bishop's equation (3.1) is a
particular case of the equation (3.5). Therefore, from the theorem of Coupet and
 Sobolev's embedding theorem  $W^{q,p}\subset C^{q-1}$, it follows that
 for  a small enough neighborhood $Q\ni 0$, $(c,t)\in Q$, the
Bishop equation (3.1) has unique solution $X(\tau,c,t):$ $ X,
h\circ X \in C^{q-1}{(\overline U \times Q)} \cap C^{q}{(U \times
Q)}.$ Note that the operator $\Im : W^{q,p} \to W^{q,p} $ is
continuous.

 Therefore, for a $C^2-$smooth generic submanifold $M \subset \C^n$,  we
obtain a smooth family of disks (3.3), attached to $M$, such that
 $$
 \left\| \Im X \right\|_1 \leq A \left\| X \right\|_1, \left\| \Im h\circ X
\right\|_1 \leq A \left\| h\circ X \right\|_1,
$$
where $A \,\,\text{is a
constant and}\,\, \left\| \cdot \right\|_1 $ is the $C^1$-norm in
$\tau \in T.$

\subsection  {Harmonic measure of boundary set of the unit disk.}  For arbitrary $\gamma \subset \T$ we put
 $\aleph (\gamma,\U)$-- class of functions
$$\left\{u(\zeta): \,u\in \text{sh}(\U)\cap
C(\overline{\U}),\,\,u|_\U<0,\,\, u|_\gamma \leq {-1}\right \},\,\,$$ and
set
$$
\omega(\zeta,\gamma,\U)=\text{sup}\{u(z): u\in \aleph\}, \ \zeta \in \U.
$$
Then (negative of) the upper semi-continuous regularisation $\omega^*(\zeta,\gamma,\U)$ is called the {\it harmonic measure} of $\gamma$ with respect to
 $\U$ at the point $\zeta$. The function $\omega^*$ is the unique solution of the Dirichlet problem:

$$
\Delta \omega^*=0 ,\,\, \omega^* |_\T=-\chi_{\gamma},
$$
where $\chi_{\gamma}$ is the chararacteristic function of $\gamma.$ By Poisson formula
$$
\omega^*(\zeta,\gamma,\U)= -\frac{1} {{2\pi}}\int\limits_\T {\chi
_{\gamma} \left( \tau  \right)\operatorname{Re} \frac{{e^{i\tau }
+ \zeta}} {{e^{i\tau }  - \zeta }}d\tau } \,\,,\,\,\,\,\zeta =
re^{i\theta }.
$$

For $\gamma=\{e^{i\varphi}: \,0\leqslant \varphi \leqslant \pi \}$ the  harmonic measure $\omega^*$ can be
expressed as follows.
\begin {equation}
\omega^*(\zeta,\gamma,\U)= \frac{1}{\pi} \, \, \text{arg}\,i\,\frac{1-\zeta}{1+\zeta} - 1.
\end {equation}
Let us define the sector at the point $1=e^{i\cdot0} \in \ove \U$ as follows 
$$
\Omega_{0,\alpha}=\cup \{ l\cap \ove {\U}: \,\,l\ni1,\,\, \pi/2 \leq \text{arg}\,{l}\leq \pi/2+\alpha\},
$$
where $l$ stands for a real line passing through the point $1$ and $ 0\leq \alpha\leq \pi/2$ is fixed. 
The sector $\Omega_{a,\alpha}$ at the point $e^{ia} \in \ove \U$ can define in the same way. From (3.6) it follows clearly that
 
\begin {equation}
\omega^*(\zeta, \gamma, \U) \leq -1+\alpha/\pi, \, \, \, \, \forall \zeta
\in \Omega_{0,\alpha}.
\end{equation}

A sector $\Omega_{a,\alpha}$  at the point $e^{ia}$  is said to be {\it admissible} if  $\,\,\Omega_{a,\alpha}\cap \partial \U \subset \gamma.$ 
From the last  fact, we deduce the following statement.
 \begin{lem} Let $\gamma=\text{arc}[e^{ia},e^{ib}] \subset \mathbb{T} ,\,\, 0 \leq a < b \leq 2\pi,$ be an arbitrary $\text{arc}$ on $\mathbb{T}$, and let $\Omega_{a,\alpha},\,\,$ be an admissible sector at the point $e^{ia}$. 
 Then $\omega^* (\zeta,\gamma,\U)$ is $\Lambda_1$-continuous in
  $\U \cup \mathbb{T} \setminus \{e^{ia}, e^{ib}\},\,\, \omega^*|_{{\gamma}^{\circ}} \equiv -1,\,\,\omega^*|_{\mathbb{T} \setminus
  {\gamma}}\equiv 0\,\,$ and $\omega^*$ satisfies (3.7) in  $\Omega_{a,\alpha}$.
  \end{lem}
Here ${\gamma}^{\circ}$ denote the interior of the arc $\gamma.$  We note that if $\gamma_0 \Subset \gamma$ in an arc with non empty interior, then there exist $\alpha=\alpha(\gamma_0,\gamma)>0$ such that $\Omega_{\tau, \alpha}$ is admissible for every $\tau \in \gamma_0.$

\section {Transversality of attached discs to a generic manifold.}

   It is clear that the  family of analytic discs  constructed above
 $$
 \Phi (c,t,\zeta) = X (c,t,\zeta) + i (h^{*} (c,t,\zeta) + t v (\zeta)),
  $$
  for $ (c,t) \in Q=Q_c \times Q_t, \zeta \in \bar \U$,
 satisfies the following properties:

 \begin{equation} \label{eq:PFE}
  X (c,t,\tau) = c - \Im \left( h \circ X (c,t,\tau) + t v(\tau) \right),  \,\,  (c,t) \in Q, \,\,  \tau \in \partial \U.
 \end{equation}

\begin{equation}
h^{^{*}}(c,t, \tau)=h \circ X (c,t, \tau), (c,t) \in Q, \tau \in
\partial \U
\end{equation}

\begin{equation}\begin{array}{ll}
X(c,0,\zeta)\equiv c, h^{^{*}}(c,0, \zeta)\equiv h(c) \,\,\,
\text{so
that} \\
\Phi(c,0, \zeta)\equiv c+ ih(c) \in M, c \in Q_c
\end{array}
\end{equation}

\begin{equation}
X(c,t,0)= \frac{1}{2\pi} \int_{T}X(c,t, \tau) d \tau \equiv c,
\,\, (c,t) \in Q.
\end{equation}

\begin{equation}
\left\| X \right\| \leq {O}(\left\| c \right\| + \left\| t
\right\|), \left\| D_\tau X \right\| \leq {O}( \left\| t
\right\|).
\end{equation}
Here and below $\left\| \cdot \right\| $  is Euclidean norm.

 The following geometric transversality property will be crucial for the proof of our main theorem.
 \begin{lem}  Let $\gamma_0\subset\subset \gamma$ an arc with non empty interior. Then for
 small enough $Q$  the attached disks $\Phi (c,t,\zeta) ,\,\, t\neq
 0,$ for $\zeta\rightarrow \tau \in \gamma_0$ meet $M$
 transversally.
 \end{lem}
 \begin{proof} $\,\,$ For the normal derivative $D_{\mathop n\limits^ \to  } $
 at the points $\tau \in \gamma_0$ we have
 $$
 \text{Im} D_{\mathop n\limits^ \to  }\Phi(c,t,\tau)=D_{\mathop n\limits^ \to
 }h^*(c,t,\tau)+tD_{\mathop n\limits^ \to  }\upsilon(\tau)
 $$
 and
 $$
 \left|\text{Im} D_{\mathop n\limits^ \to  } \Phi(c,t,\tau)\right|
 \geq  \,\,\parallel { t}\parallel b\, -\,O(\varepsilon)\parallel {
 t}\parallel= \parallel { t}\parallel(b-O(\varepsilon)),
 $$
 where
 $$
 b := \mathop {\inf }\limits_{\gamma_0}  \left|
{D_{\mathop n\limits^ \to  } v\left( \tau  \right)} \right| > 0 \,\, \text{and}\,\,\, \varepsilon
=\text{sup}\left\{\left\|c\right\|+ \left\|t\right\|: \,\,c\in
Q_c,\,t\in Q_t \right\}.
$$

It follows, that for $O(\varepsilon) < \frac{b} {2}$
 \begin{equation}
 \left|\text{Im} D_{\mathop n\limits^ \to  }
\Phi(c,t,\tau)\right| \geq  \,\,\parallel { t}\parallel b/2\, \,\,\,\,
\forall \tau \in \gamma_0,
\end {equation}
i.e. the disks $\Phi(c,t,\zeta)$ meet $M$ for $\zeta\rightarrow \tau \in \gamma_0$ transversally.
\end{proof}

\begin{cor} Let $Q'=  \{ \parallel { t}\parallel=
  \sigma \} \subset Q_t ,$ where $\sigma> 0.$ Then there exist a neighborhood $\Omega\,' \supset \gamma_0$ and a constant
   $C>0$ such that
 \begin {eqnarray}
  d\,_\C (\zeta,\gamma_0) & \leq & C d_{\C^n} [\Phi(c,t,\zeta ),M] \,\,\,\\
  d_{\C^n} [\Phi(c,t,\zeta ), \Phi(c,t,\gamma_0 )] & \leq & C d_{\C^n} [\Phi(c,t,\zeta ),M],
  \nonumber
  \end{eqnarray}
  $\forall \,\,\zeta \in \Omega =\overline {\U \cap \Omega'},\,\, t\in Q' , c\in \overline{Q}_c.$
  Here $d\,_\C$ and $d_{\C^n}$ are Euclidean distances on $\C$ and $\C^n$, respectively.
 \end{cor}
  \begin{proof}  The statement clearly follows from (4.6), because for every fixed $t^0\in
   Q',\,\, c^0 \in \overline Q_c$ we can write  (4.7), which then will be  true in some neighborhoods $B_c \ni c^0,\, B_t \ni
  t^0.$
\end{proof}

\begin{lem}  For every $\Omega'\supset
\gamma_0\,$ and for every $Q'_t=  \{\parallel { t}\parallel
=\sigma \} \subset Q_t ,$ $\sigma > 0$ small enough the
 closed set $\,\,W=\{ \Phi(c,t,\zeta ) \in \C^n : c\in \overline Q_c,\, t\in Q'_t, \zeta \in \Omega =\overline {\U  \cap \Omega'} \}$
 contains the point $0\in M$ in its interior in $\C^n$ i.e. $0\in \dot{W}$.
 \end{lem}

 \begin{proof} By (4.3)  $X(c,t,\zeta) \equiv c$ if $t=0.$ Since
 $X$ is smooth, then for small enough fixed $t^0$ and for arbitrary
 fixed $\tau^0 \in \gamma_0$  the image $X(c,t^0,\tau^0):\, c\in Q_c
 $ contains $0\in \R^n.$ It follows, that $0\in W.$ Moreover, $\dot{W} \neq \emptyset$  and if for some
  $\|t^0\|\leq \sigma,\, \zeta^0 \in \ove U $
 \begin{equation}
 x(c, t^0, \zeta^0) \in \frac{1}{2} Q_c,  \,\,\text{then} \,\, c\in Q_c
 \end{equation}

 Now we assume by contradiction that $0\in \partial W.$ Then  $\C^n \setminus W$ is open and contains $0$ on its boundary.
 It is clear, that near $0$ there exists a  point $p^0 =(x^0,y^0)\in {\partial \dot{W}\setminus M}$ such,
 that $ x^0 \in \frac{1}{2} Q_c$ and  $p^0= \Phi(c^0,t^0,\zeta^0)$ for some $\,\,c^0\in \ove{Q}_c, \,\,
 \parallel { t^0}\parallel
  =\sigma, \,\,\zeta^0\in \Omega' \cap U.$

 For simplicity we may assume that $ t^0=(0,...,0,\sigma)$  and set
  $ 'c = \left(c_1 ,...,c_{n - 1}\right),\,\,'t = \left(t_1 ,...,t_{n - 1}\right)$.
 From  (4.8) it follows also, that $c\in Q_c.$

We consider the transformation
\begin{eqnarray}
S \left({'c} ,{'t},\zeta\right) = \Phi
\left({'c},c_n^0,{'t},t_n^0,\zeta\right): \,\, {'Q} \times
\overline \U  \longrightarrow  \C^n,
\end{eqnarray}
where $ 'Q :=  \{z\in Q: \,\,c_n = c_n^0, t_n = t_n^0\} \subset
\R^{2n - 2}.$

 Then $ S\left( {'c^0,'t^0 ,\zeta ^0}\right) = p^0$ and its Jacobian is given by
  $$
  J('c,'t,\zeta)=\hat{J}(c,t,\zeta)|_{c_n=c^0_n, t_n=t^0_n},
  $$
 where
 $$
\hat{J}(c,t,\zeta) = \left| {\begin{matrix} {\frac{{\partial X_1
}} {{\partial c_1 }}} & {...} & {\frac{{\partial X_{n - 1} }}
{{\partial c_1 }}} & | & {\frac{{\partial Y_1 }} {{\partial c_1
}}} & {...} & {\frac{{\partial Y_{n - 1} }} {{\partial c_1 }}} & |
& {\frac{{\partial X_n }} {{\partial c_1 }}} & {\frac{{\partial
Y_n }}
{{\partial c_1 }}}  \\
   {...} & {...} & {...} & | & {...} & {...} & {...} & | & {...} & {...}  \\
   {\frac{{\partial X_1 }}
{{\partial c_{n - 1} }}} & {...} & {\frac{{\partial X_{n - 1} }}
{{\partial c_{n - 1} }}} & | & {\frac{{\partial Y_1 }} {{\partial
c_{n - 1} }}} & {...} & {\frac{{\partial Y_{n - 1} }} {{\partial
c_{n - 1} }}} & | & {\frac{{\partial X_n }} {{\partial c_{n - 1}
}}} & {\frac{{\partial Y_n }}
{{\partial c_{n - 1} }}}  \\
   { -  - } & { -  - } & { -  - } & | & { -  - } & { -  - } & { -  - } & { -  - } & { -  - } & { -  - }  \\
   {\frac{{\partial X_1 }}
{{\partial t_1 }}} & {...} & {\frac{{\partial X_{n - 1} }}
{{\partial t_1 }}} & | & {\frac{{\partial Y_1 }} {{\partial t_1
}}} & {...} & {\frac{{\partial Y_{n - 1} }} {{\partial t_1 }}} & |
& {\frac{{\partial X_n }} {{\partial t_1 }}} & {\frac{{\partial
Y_n }}
{{\partial t_1 }}}  \\
   {...} & {...} & {...} & | & {...} & {...} & {...} & | & {...} & {...}  \\
   {\frac{{\partial X_1 }}
{{\partial t_{n - 1} }}} & {...} & {\frac{{\partial X_{n - 1} }}
{{\partial t_{n - 1} }}} & | & {\frac{{\partial Y_1 }} {{\partial
t_{n - 1} }}} & {...} & {\frac{{\partial Y_{n - 1} }} {{\partial
t_{n - 1} }}} & | & {\frac{{\partial X_n }} {{\partial t_{n - 1}
}}} & {\frac{{\partial Y_n }}
{{\partial t_{n - 1} }}}  \\
   { -  - } & { -  - } & { -  - } & | & { -  - } & { -  - } & { -  - } & | & { -  - } & { -  - }  \\
   {\frac{{\partial X_1 }}
{{\partial \zeta '}}} & {...} & {\frac{{\partial X_{n - 1} }}
{{\partial \zeta '}}} & | & {\frac{{\partial Y_1 }} {{\partial
\zeta '}}} & {...} & {\frac{{\partial Y_{n - 1} }} {{\partial
\zeta '}}} & | & {\frac{{\partial X_n }} {{\partial \zeta '}}} &
{\frac{{\partial Y_n }}
{{\partial \zeta '}}}  \\
   {\frac{{\partial X_1 }}
{{\partial \zeta ''}}} & {...} & {\frac{{\partial X_{n - 1} }}
{{\partial \zeta ''}}} & | & {\frac{{\partial Y_1 }} {{\partial
\zeta ''}}} & {...} & {\frac{{\partial Y_{n - 1} }} {{\partial
\zeta ''}}} & | & {\frac{{\partial X_n }} {{\partial \zeta ''}}} &
{\frac{{\partial Y_n }}
{{\partial \zeta ''}}}  \\
 \end{matrix} } \right|,
$$
Here $\zeta= \zeta' + i \zeta''$ and $
 Y_k (c,t,\zeta) = h_k^* \circ X\left( c ,t,\zeta \right) + t_k v\left(
\zeta\right), \ \ k = 1, \cdots n . $

The determinant $J$,  is composed by $9$ block matrices $D_{ij}
\,\,,\,\,i,j = 1,2,3$.

We will show that $ J ('c^0,'t^0,\zeta^0) \neq 0$, which will imply that  the operator  $S$ is a local
 diffeomorphism in a neighborhood of the point $\left( {'c}^0 ,{'t}^0 ,\zeta^0\right)\,$.

   Indeed, by (4.3)  $
X\left(c,0,\zeta\right) \equiv c\,\,,\,\,h^{*}\left( {c,0,\zeta}
\right) \equiv h\left( c \right) $ and then

$$
 \left| {\begin{matrix}
  {D_{1 1}} &  {D_{1 2}}  \\
   {D_{2 1}} & D_{2 2}  \\
  \end{matrix} } \right|_{(c,0,\zeta)} =  D_{11} .D_{22}  = v^{n - 1} \left( {\zeta } \right)
$$
and
\begin{equation}
\left| {\begin{matrix}
  {D_{1 1}} &  {D_{1 2}}  \\
   {D_{2 1}} & D_{2 2}  \\
  \end{matrix} } \right|_{(c,t,\zeta)} =  v^{n - 1} \left( {\zeta } \right) +
  O\left(\varepsilon\right),
  \end{equation}
  where we recall that $\varepsilon
={sup}\left\{\left\|c\right\|+ \left\|t\right\|: \,\,c\in
Q_c,\,t\in Q_t \right\} .$
 Note also that

$$
D_{3 3} =  \left| {\begin{matrix}
  {\frac{{\partial X_n }}{{\partial \zeta' }}} &  {\frac{{\partial Y_n }}
{{\partial \zeta'}}}  \\
   {\frac{{\partial X_n }}
{{\partial \zeta''}}} & {\frac{{\partial Y_n }}
{{\partial \zeta'' }}}  \\
  \end{matrix} } \right| = \left|\frac{ d }{d \zeta}(X_n + i Y_n)\right|^2.
$$
 Now consider the right hand side near the arc $\gamma $. It is clear that for every $s > 0,$ there is an
 open set $\tilde{\Omega} \supset \gamma_0$ such that
 \begin{equation}
 \left|\frac{ d }{d \zeta}(X_n + i Y_n) (c,t,\zeta) \right|^2 \geq \left| D_{\tau} X_n (c,t,\tau)\right|^2 -
 s,
 \forall \zeta \in U \cap \tilde{\Omega}, \tau \in \gamma_0.
 \end{equation}

   We calculate  $
D_\tau  X\left( {c,t,\tau } \right)\,\,,\,\,$ for $\left({c,t} \right)
\in Q\,\,,\,\,\tau  \in T,$
 \begin{equation}
D_\tau  X\left( {c,t,\tau } \right) =  - D_\tau  \Im h \circ
X\left( {c,t,\tau } \right) - tD_\tau  \Im v\left( \tau  \right)
\end{equation}

Since, $ D_\tau  \Im v\left( \tau  \right) = D_{\mathop n\limits^
\to  } v\left( \tau  \right) $, where   $ D_{\mathop n\limits^ \to
} $
  is the normal derivative $
\mathop n\limits^ \to $, then (4.12) implies
$$
D_\tau  X\left( {c,t,\tau } \right) + tD_{\mathop n\limits^ \to  }
v\left( \tau  \right) = \Im D_\tau  h \circ X\left( {c,t,\tau }
\right).
$$
For   $k$-coordinate of  vector $X\left( \tau  \right) = X\left(
{c,t,\tau } \right) $
 we have
 \begin{equation}\begin{array}{ll}
\left\| {D_\tau  X_k \left(c,t, \tau  \right) + t_{k}D_{\mathop
n\limits^ \to  } v\left( \tau  \right)} \right\| = \left\| {\Im
D_\tau  h_{k} \circ X\left(c,t, \tau  \right)}
\right\|  \leqslant \\
 \leqslant {const}\left\|
{D_\tau  h_{k} \circ X\left(c,t, \tau  \right)} \right\| \leqslant
O\left( \varepsilon  \right)\left\| {D_\tau X\left(c,t, \tau
\right)} \right\|.
\end{array}
\end{equation}

Therefore,
\begin{equation}\begin{array}{ll}
\left| {t_k D_{\mathop n\limits^ \to  } v\left( \tau \right)}
\right| - O\left( \varepsilon  \right)\left\| t \right\| \leqslant
\left| {D_\tau  X_k \left( {c,t,\tau } \right)} \right| \leqslant
 \\ \leqslant \left| {t_k D_{\mathop n\limits^ \to  } v\left( \tau
\right)} \right| + O\left( \varepsilon  \right)\left\| t
\right\|\,\,,\,\,1 \leqslant k \leqslant n\,\,,\,\,\tau  \in T
\end{array}
\end{equation}

The second part of (4.14) implies
 \begin{equation}
\left\| {D_\tau  X\left( {c,t,\tau } \right)} \right\| \leqslant
C\left\| t \right\|\,\,,\,\,\left( {c,t,\tau } \right) \in Q
\times T\,\,,\,\,C - \text{constant}
\end{equation}

As in Lemma 4.1 if  $b = \mathop {\inf }\limits_{\gamma_0} \left|
{D_{\mathop n\limits^ \to  } v\left( \tau  \right)} \right|
> 0 $ and $O(\varepsilon) < \frac{b} {2},$
  then
 the first part of $(4.14)$ implies
\begin{equation}
 \vert D_{\tau} X_k (c,t,\tau)\vert \geq \vert t_k\vert b - \Vert t\Vert b \slash 2,
\end{equation}
for $\tau \in \gamma, 1 \leq k \leq n$.

 By $(4.10)$ and $(4.11)$  it follows that
 \begin{eqnarray*}
 \left|J ('c,'t,\zeta)\right| &= & \left|D_{1 1}|\cdot |D_{2 2}\right| \cdot \left|\frac{ d }{d \zeta}(x_n + i y_n)
  ('c,'t,\zeta)\right|^2 + O (\varepsilon) \\
 &\geq & \left[v^{n - 1} (\zeta) + O (\varepsilon)\right] \cdot \left[|t_n b \slash 2|^2
   - s \right] + O (\varepsilon),
 \end{eqnarray*}
 for all $('c,'t,\zeta) \in {'Q} \times \left[\U \cap \Omega'\right],$ because $\parallel { t^0}\parallel =|t^0_n|.$

 We can take $ \tilde{\Omega} \cap \Omega'$ instead of $\Omega $ and observe that all functions $O (\cdot)$ do not
 depend  on $\zeta.$
 Therefore if we take $\varepsilon,  s$ small enough, then $ \left|J ('c^ 0,'t^ 0,\zeta^ 0)\right| > 0.$

 Since, the plane $\{t_n=t^0_n\}$   is tangent to the sphere $\|t\|=\sigma$ at the point $t^0$,  the Jacobian of the
 restriction
  $ \breve{S}=\Phi \left('c, c^0_n, \sqrt{|t_1|^2+...+|t_{n-1}|^2}, \zeta \right)$ also
  is not zero at the point $('c^ 0,'t^ 0,\zeta^ 0).$
 In particular, the operator $$
 \breve{S}: \,\,U_1 \times U_2 \times U_3 \rightarrow U (p)
 $$
is a homeomorphism, where $ U_1 \subset \R^{n-1}-$ a neighborhood
of the point $'c^0$,  $ U_2 \subset Q'_t-$ a neighborhood of  $t^0
\in Q'_t$ and $U_3= \{\vert \zeta - \zeta_0\vert <
\sigma'\}\subset \Omega, \sigma'
> 0,$ is a neighborhood of $\zeta_0$.  It follows, that the open set $U(p) \subset W,$ that is contradiction to
$p \in \partial W. $
\end{proof}

\section  {Proof of the main Theorem}

First we observe that from the results of Edigarian-Wiegerinck [EW10]
and the authors [SZ12], it follows that $M  \subset \C^n$ is a
pluriregular set.  Indeed, it was proved in [SZ12] that a set of
full measure in a generic manifold $M$ is not thin. Since the set
$P=\{z\in M: V^*_M(z)>0 \}$, where $V^*_M(z)$ is Green function, is pluripolar by Bedford and Taylor ([BT82]), it
has zero-measure (see [Sa76, C92]) and then  the set $M \setminus
P$ is not thin. Therefore $V^*_M\equiv 0$ on $M$, i.e. $M$ is
pluriregular. Note that in [EW10] non-thinness of $M\setminus P$
was proved for $C^1$-smooth manifold $M$ and for a pluripolar set
$P\subset M$, which implies that an arbitrary $C^1$-smooth generic
manifold is pluriregular.

 Our main theorem will be a consequence of the following result, thanks to Lemma  2.3.

 \begin{thm}
 Any  $C^2$-smooth generic submanifold  $M \subset \C^n$ is $\Lambda_1$-pluriregular.
 \end{thm}
 \begin{proof}
 We first reduce to the case of a totally real submanifold.
 Fix a point, say $z^0 = 0 \in M$.
 Changing holomorphic coordinates in $\C^n$, we can assume that the tangent space $T_0 M$,
 which by definition does not contain any complex
 hyperplane, can be written as
 $$
 T_0 M = \{z = x + i y \in \C^n : y_1 = \cdots = y_{2 n - m} = 0\}.
 $$
 Hence for a small neighborhood $G = G_1 \times G_2$ of the origin with
 $$G_1 =\{ (x,y'') = (x,y_{2 n - m + 1}, ... , y_n) \in \R^n \times \R^{m - n} : \vert x\vert \leq \delta,
 \vert y''\vert < \delta \},$$
 $$G_2 =\{ y' = (y_1, \cdots, y_{2n - m}) \in \R^{2 n - m} : \vert y'\vert < \delta\},$$

 we can represent $M$ as a graph
 $$
  M \cap G = \{z  \in G :  y' = h (x,y'') \},
 $$
 where $h$ is $C^2$ smooth  mapping from $G_1$ into $G_2$.

 Observe that  for each small enough $y''_0$ the intersection $M\cap\Pi\{y''_0\}$ of $M$ with the plane
 $\Pi\{y''_0\} := \{ z \in \C^n : y'' = y''_0\}$ is
 an $n-$dimensional generic manifold. Moreover, since the Green function is monotonic, i.e.
  $V(z,E_1)\geq V(z,E_2)$ for $E_1 \subset E_2,$ it is enough to prove the theorem in the case when $M$ is generic of dimension $n$, hence totally real of dimension $n$.

  In this case, we show the local H\"older pluriregularity of $M$, using previous results from Section 4. Fix  a point $p  \in M$, a ball $B(p) = B_x \times B_y \subset \C^n$ centered at the point $p$
such that $ M_p := M \cap B (p)$ is the graph of a $C^2$-smooth
function. Then by Corollary 4.2, for arbitrary fixed
small $\sigma >0$  there exist a neighborhood   $\Omega' \supset
\gamma_0$ and a constant $C>0$, depending on the point $p$, such
that the inequalities (4.7) hold. By Lemma 4.3  $\,\, O (p) \ni p$, where $O (p) = W^0$ is the interior  of the set
$W=W(p,\Omega', \sigma),$ constructed in  Lemma 4.3.

Fix  a point  $z^0 \in O (p) \setminus M $ and a disk
   $ \Phi(c,t,\zeta ):\,\,\Phi(c^0,t^0,\zeta^0)=z^0,\,\,$ with $c^0 \in \overline Q_c,\, t^0\in Q'_t, \zeta^0 \in \U \cap
   \Omega'$. Then the function $V_{M_p}\circ \Phi(c^0,t^0,\zeta) \in sh (\U)$ and
   $V_{M_p}\circ \Phi|_\gamma\equiv 0$. Let $C'' =\mathop {\max
}\limits_{B (p)} V_{M_p}(z) <\infty.$ By the theorem of two constants we have
\begin{equation}
 V_{M_p}\circ \Phi (c^0,t^0,\zeta)\leq
 C'' [\omega^*(\zeta,\gamma,\U)+1],\,\,\zeta\in \U.
\end{equation}

Therefore the first part of  Lemma 3.1 and (4.7) yields the
inequality

$$ V_{M_p}(z^0) = V_{M_p} \circ \Phi (c^0,t^0,\zeta^0)\leq C'' \left[\omega^*(\zeta^0,\gamma,\U)+1)\right] \leq  $$
\begin{eqnarray}
 \leq C' C'' d_{\C}(\zeta^0,\gamma_0)\leq C_p d_{\C^n}(z^0,M_p),
\end{eqnarray}
for all $z^0 \in O(p)$, where $C_p := C C' C''$ depends on the
fixed point $p \in M$ and on the corresponding family of analytic
discs, attached to $M$ locally, in a neighborhood of $p.$

Now given a compact set $K \subset M$ we can apply the previous estimate to each point of $K$. Then by compactness we
 can find a finite number of points $p_1, \cdots p_k$ of $K$, a finite number of balls  balls $B (p_1), \cdots, B(p_k)$ and a finite numbers of open sets $ O (p_1), \cdots O (p_k)$ such that

 \begin{eqnarray*}
V_{M_p} (z) \leq C_p d_{\C^n}(z,M_p),
\end{eqnarray*}
for any $z \in O (p)$ and $p = p_1, \cdots , p_k$.
Now observe that $O = \cup_{1 \leq i \leq k} O (p_i)$ is a neighborhood of $K$ and shrinking a little bit the open
sets $O_p$ we can assume that
for any  $p=p_i$ and $z \in O_p$ , $d_{\C^n}(z,M_p)  \leq  d (z, M).$
Since $V_M \leq V_{M_p}$, it follows that
$
V_{M} (z) \leq A  d_{\C^n}(z,K),
$
for any $z \in O$.$\vartriangleright$
\end{proof}

\section{Open problems }

Let $D\subset M$ a domain with $C^1$-smooth boundary $\partial D.$
   Lemma 4.2 states that a neighborhood of the generic manifold locally consists in the  interior $\dot{W}$ of the set
 $W=\{ \Phi(c,t,\zeta ): c\in \overline Q_c,\, t\in Q'_t, \zeta \in \Omega =\overline {\U  \cap \Omega'} \}.$
 It seems clear, at least intuitively, that if here, instead of $\Omega$, we take its part $\Omega_{a,\alpha},\,\,\alpha>0$ (see Lemma 3.1), then
 we should see that $\dot{W}$ contains some wedge $$\left\{z\in \C^n: \,\, d_{\C^n}\left(z, M \right)< C_\alpha \cdot d_{\C^n}\left(z, \partial D
 \right)\right\},$$ where $C_\alpha >0$ is a constant. If this is true then we could prove: {\it
  arbitrary  close $C^1$-domain $\overline D$ in $C^2$-smooth generic manifold is pluriregular, i.e. the Green function $V^*(z, \overline
  D)$ is continuous  in} $\, \C^n.$
  The proof easily follows  by the well-known criteria of pluriregularity (see [Sa80]) and by the following lemma.
\vskip 0.3 cm
 \textbf {Lemma 6.1.} {\it If $f(\lambda)$ is a $C^1$-smooth funstion on $[0,1] \subset \R$, then for every $\varepsilon >0$ there exist polynom $p(\lambda)$ such that}
 $$
   | p(\lambda)-f(\lambda)| \leq \varepsilon \lambda,\,\, \lambda\in [0,1].
 $$

  The authors do not know any proof of the following

 \vskip 0.3 cm
\textbf {Conjecture:} Let $ D \subset M $
 be a bounded domain in $M$ with smooth boundary. then $\overline{D}  \subset \C^{n}$ is $\Lambda_{1 \slash 2}\,$-pluriregular i.e.
  its pluricomplex Green function $V(z,\overline{D})$ is H\"older continuous of order $1 \slash 2$ in $\C^n$.
 \vskip 0.3 cm
  We note  that if $M$ is real analytic generic manifold then the conjecture is true.

                     \begin{center}\textbf{  References}

                          \end{center}
\vskip 0.3 cm

 \noindent [B65] E.Bishop: {\it Differentiable manifolds in complex Euclidean spaces}. Duke Math.J., V.32 (1965), no.1, 1-21.

\noindent [BT82] E.Bedford, B.A. Taylor: {\it A new capacity for
plurisubharmonic functions}. Acta Math., V.149  (1982), no. 1-2,
 1-40.

 \noindent[C92] B. Coupet: {\it Construction de disques analytiques et
    r\'egularit\'e de fonctions holomorphes au bord}. Math.Z., V.209
    (1992), no.2, 179-204.

\noindent [DNS]  T.-C. Dinh, V.-A. Nguyn, N. Sibony: {\it Exponential estimates for plurisubharmonic functions and stochastic dynamics}. J. Differential Geom. V. 84 (2010), no. 3, 465-488.

 \noindent [EW10] A. Edigarian and J. Wiegerinck: {\it  Shcherbina's theorem for finely holomorphic functions}.
    Math. Z., V.266 (2010), no.2, 393-398.

  \noindent [FS92] J.E.Fornaes, N.Sibony:{\it Complex Henon mappings in $\C^2$ and Fatou-Bieberbach domains}.
Duke Math .J., V. 65 (1992), 345-380.

 \noindent [GZ07] V.  Guedj and A. Zeriahi, {\it The weighted Monge-Amp\`{e}re energy of quasiplurisubharmonic functions}. J. Funct. Ann. V.250 (2007), 442-482.

 \noindent[HC76] G. M. Henkin and E. M. Chirka: {\it Boundary propertie of holomorphic functions of several variables}.
 J. Math. Sci., V.5 (1976), 612-687.

 \noindent [Kl91] M.Klimek: {\it Pluripotential theory.} London Mathematical Society Monographs. New Series, 6. Oxford Science Publications. The Clarendon Press, Oxford University Press, New York, 1991.

 \noindent[Kol98] S. Kolodziej: {\it The complex Monge-Amp\`ere equation }.
   Acta Math.,  V.180 (1998), 69-117.

 \noindent[Kos97] M. Kosek: {\it H\"older continuity property of filled-in Julia sets in $\C^n$}. Proc. of the
AMS.,V.125 (1997), no. 7, 2029-2032.

 \noindent[PP86] W.Pawlucki, W.Plesniak: {\it Markov's inequality and $C^ {\infty}$ functions on sets with polynomial
cusps.} Math. Ann., V. 275 (1986), no. 3, 467-480.

 \noindent[Pi74] S.Pinchuk: {\it A Boundary-uniqueness theorem for holomorphic functions of several complex
 variables}. Math. Zametky, V.15 (1974),no.2, 205-212

 \noindent[Sa76] A. Sadullaev: {\it A boundary-uniquiness theorem in $\C^n$}. Mathematical Sbornic, V.101(143)(1976),no.4,
 568-583= Math. USSR Sb., V.30 (1976), 510-524.

 \noindent[Sa80] A. Sadullaev: {\it  P-regularity of sets in $C^n$}. Lect. Not. In Math., V.798 (1980),
402-407.

  \noindent[Sa81] A. Sadullaev: {\it Plurisubharmonic measure and capacity on complex manifolds}. Uspehi Math. Nauk,
V.36(220)(1981, no.4, 53-105=
  Russian  Math. Surveys V.36 (1981), 61-119.

  \noindent[SZ12] A.Sadullaev, A.Zeriahi: {\it Subsets of full measure in a generic submanifold in $\C^n$ are non-plurithin}. 	Math. Z., Vol.274 (2013), 1155-1163.

\noindent[Si62] J. Siciak: {\it On some extremal functions and
their applications in the theory of analytic functions of several
complex variables}. Trans. Amer. Math. Soc., V.105 (1962),
322-357.

\noindent[Si81]  J. Siciak: {\it Extremal plurisubharmonic
functions in ${\bf C}^{n}$}. Ann. Polon. Math., V.39 (1981),
175-211.

\noindent[Si97]  J. Siciak: {\it Wiener's type sufficient
conditions in $C^N$}. Univ. Iagel. Acta Math., V.35 (1997), 47-74.

  \noindent[Za74] V.P. Zahariuta:  {\it Extremal plurisubharmonic functions, Hilbert scales, and the isomorphism of spaces of analytic functions of several variables.} I, II. (Russian) Teor. Funkciy Funkcional. Anal. i Prilojenie, V.19 (1974), 133-157;   V.21 (1974), 65-83.

\noindent[Za76] V.P. Zahariuta: {\it Extremal plurisubharmonic
functions, orthogonal polynomials and Bernstein-Walsh theorems for
analytic functions of several complex variables}. Ann. Polon. Math. V.33 (1976/77), no. 1-2, 137-148.

 \noindent[Ze87] A. Zeriahi: {\it Meilleure approximation polynomiale et croissance des fonctions enti\` eres
 sur certaines vari\'et\'es alg\'ebriques affines}.  Ann. Inst. Fourier (Grenoble), V.37 (1987), no.2, 79-104.

 \noindent[Ze91]. A. Zeriahi:  {\it Fonction de Green pluricomplexe \` a p\^ole \`a l'infini sur un espace de
 Stein parabolique et applications.} Math. Scand., V.69 (1991), no.1, 89-126.

\noindent[Ze93] A. Zeriahi:  {\it In\'egalit\'es de Markov et d\'eveloppement en s\'erie de polyn\^omes orthogonaux des
 fonctions $C^{\infty}$ et $A^{\infty}$.} Several complex variables (Stockholm, 1987/1988), 683-701, Math. Notes, 38,
Princeton Univ. Press, Princeton, NJ, 1993.

\noindent[Ze01] A. Zeriahi: {\it Volume and capacity of sublevel sets of a Lelong class of plurisubharmonic functions.}
 Indiana Univ. Math. J., 50 (2001), no.1, 671-703.
\\

\noindent A.Sadullaev, National University of Uzbekistan, \\
100174 Tashkent, Uzbekistan\\
{sadullaev@mail.ru} \\

\noindent A. Zeriahi, Institut de Math\'ematiques de Toulouse \\
Universit\'e Paul Sabatier, 118 Route de Narbonne, \\
31062 Toulouse \\
zeriahi@math.univ-toulouse.fr

 \end{document}